\documentclass[11pt]{elsarticle}

\usepackage{amsthm}
\usepackage{amsmath}
\usepackage{amssymb}

\newtheorem{theorem}{Theorem}[section]
\newtheorem{prop}[theorem]{Proposition}
\newtheorem{lemma}[theorem]{Lemma}

\numberwithin{equation}{section}

\newcommand{\cart}{\mathcal{C}}

\newcommand{\dist}{\mathcal{D}}
\newcommand{\eval}[2]{\left. #1 \right|_{#2}}
\newcommand{\FN}{Fr\"{o}licher--Nijenhius}
\newcommand{\fnb}[1]{[\![ #1 ]\!]}
\newcommand{\force}{\mathcal{F}}
\newcommand{\gh}{\hat{g}}
\newcommand{\horproj}{\mathcal{P}}
\newcommand{\identity}{\mathcal{I}}
\newcommand{\Lie}{\mathcal{L}}
\newcommand{\ombar}{\bar{\omega}}
\newcommand{\omdot}{\dot{\omega}}
\newcommand{\p}{\partial}
\newcommand{\pd}[2]{\frac{\p #1}{\p #2}}
\newcommand{\pdb}[3]{\frac{\p^2 #1}{\p #2 \p #3}}
\newcommand{\pdeform}{\mathcal{G}}
\newcommand{\Phit}{\tilde{\Phi}}
\newcommand{\R}{\mathbb{R}}
\newcommand{\Rt}{\tilde{R}}
\newcommand{\slice}[1]{{}^{(#1)}\!}
\newcommand{\vertproj}{\mathcal{Q}}
\newcommand{\vertprojt}{\tilde{\vertproj}}
\newcommand{\vf}[1]{\frac{\p}{\p #1}}
\newcommand{\vl}{\mathrm{v}}
\newcommand{\yddot}{\ddot{y}}
\newcommand{\ydot}{\dot{y}}

\DeclareMathOperator{\annih}{annih}
\DeclareMathOperator{\curv}{R}
\DeclareMathOperator{\id}{id}
\DeclareMathOperator{\rank}{rank}
\DeclareMathOperator{\Span}{span}

\newcommand{\art}[6]{#1: #2 {\it #3\/} {\bf #4} (#5) #6}
\newcommand{\book}[4]{#1: {\it #2\/} (#3, #4)}
\newcommand{\inbook}[7]{#1: #2 {\it In:\ #3\/}, ed. #4 (#5, #6) #7}

\begin{document}

\begin{frontmatter}

\title{Tangent bundle geometry induced by second order partial differential equations}

\author[djs]{D.J.\,Saunders\corref{cor}}
\ead{david@symplectic.demon.co.uk}

\author[or]{O.\,Rossi\fnref{fnor}}
\ead{olga.rossi@osu.cz}

\author[gep]{G.E.\,Prince\fnref{fngep}}
\ead{g.prince@latrobe.edu.au}

\cortext[cor]{Corresponding author}
\fntext[fnor]{and Department of Mathematics, Ghent University, Belgium}
\fntext[fngep]{and The Australian Mathematical Sciences Institute, c/o The University of Melbourne, Victoria 3010, Australia}

\address[djs]{Department of Mathematics, Faculty of Science, University of Ostrava,\\ 30.\,dubna 22, 701 03 Ostrava, Czech Republic}
\address[or]{Department of Mathematics, Faculty of Science, University of Ostrava,\\ 30.\,dubna 22, 701 03 Ostrava, Czech Republic}
\address[gep]{Department of Mathematics and Statistics, La Trobe University, Victoria 3086, Australia}


\begin{abstract}
We show how the tangent bundle decomposition generated by a system of ordinary differential equations may be generalized to the case of a system of second order PDEs `of connection type'. Whereas for ODEs the decomposition is intrinsic, for PDEs it is necessary to specify a closed $1$-form on the manifold of independent variables, together with a transverse local vector field. The resulting decomposition provides several natural curvature operators. We give three examples to indicate possible applications of this theory.\\[2ex]
\end{abstract}

\end{frontmatter}


\section{Introduction}

One of the important developments in the understanding of systems of second order ordinary differential equations (expressed in solved form) has been the use of an associated `non-linear connection' and its curvature components. For example, this approach has enabled significant advances to be made in understanding the inverse multiplier problem in the calculus of variations, a problem solved in the case of two equations (using analytic techniques) by Douglas in 1941~\cite{Dou41}, but still outstanding where there are three or more equations:\ see~\cite{AnTh92,KrPr08,STP02} for expositions of this approach to the problem. Other problems where the non-linear connection and its linearized version have been of benefit include the question of whether a system of second order ODEs can be decoupled in some system of coordinates, or more generally is `submersive'~\cite{KoTh91}; and whether a congruence of curves given by a system of second order ODEs will become degenerate~\cite{JePr2000}. Our aim in this paper is to apply the same approach to a particular class of systems of second order partial differential equations. We shall see that, in principle, a similar construction may be performed, but that there are many more subtleties to be considered. In particular, we find a number of different curvature operators which may be used to analyse properties of the original equations.  

We shall briefly summarise the methods used in the ODE case. Given a system of autonomous equations, expressed geometrically as a vector field $\Gamma$ of suitable type on some tangent manifold $TM$, the non-linear connection has typically been constructed by taking advantage of the `almost tangent structure', or `vertical endomorphism', a vector-valued $1$-form $S$ always found on a tangent manifold~\cite{ClBr60,Cra83}. Taking the Lie derivative $\Lie_\Gamma S$ gives another vector-valued $1$-form which has two eigenvalues, $+1$ and $-1$. The $+1$-eigenspaces form the vertical sub-bundle of the repeated tangent bundle $TTM \to TM$, whereas the $-1$-eigenspaces form a complementary distribution:\ this is the non-linear connection.

A similar technique can be used for time-dependent second order ODEs. The structure now is of a manifold fibred over $\R$, the manifold of time values, and a vector field $\Gamma$ of suitable type on the first jet manifold of this fibration. Once again there is a vertical endomorphism $S$, but now $\Lie_\Gamma S$ has a third eigenvalue, zero, with corresponding one-dimensional eigenspaces spanned by the values of $\Gamma$ itself.

There is, interestingly, a second approach to the construction of the horizontal distribution which may be used in both the autonomous and time-dependent cases. This second approach regards the system of equations, not as a vector field, but as a submanifold of a second order structure:\ the second order tangent manifold, or second order jet manifold, as the case may be. There is now a second order version of the vertical endomorphism (see~\cite{CSC86} for the autonomous case), and this can be applied to the annihilator of the equation's tangent sub-bundle. Pulling this structure back to the first order manifold (tangent or jet) then gives the same construction but in dual form, as a decomposition of a cotangent bundle.

For the systems of second order PDEs studied in this paper, we wish to carry out analogues of these constructions, and so we again want the equations to be written `in solved form'. This now suggests that we restrict attention to those particular systems of over-determined PDEs which can be described using sections of the jet projection $J^2 \to J^1$:\ they are equations of connection type, and take the form
\begin{equation}
\label{Econnpde}
\pdb{y^\sigma}{x^i}{x^j} = F^\sigma_{ij} \biggl( x^k, y^\nu, \pd{y^\nu}{x^k} \biggr) \qquad 1 \le \sigma, \nu \le m, \quad 1 \le i,j,k \le n.
\end{equation}
We shall use analogues of the vertical endomorphism operators appropriate for the case of several independent variables, and use both the approaches described above. Our principal result is that a tangent bundle decomposition can be achieved when given a closed $1$-form $\varphi$ on the manifold of independent variables, together with a local vector field $v$ on that manifold satisfying $i_v\varphi \ne 0$. The detailed consequences are, however, quite unexpected, because the three-way splitting already known in the case of ODEs generalizes in two different ways:\ one arises directly from the eigenvalue problem for the Lie derivative of the vertical endomorphism, with positive, negative and zero eigenvalues, whereas the other comes from an alternative decomposition of the direct sum of the positive and zero eigendistributions. These two splittings may then be combined to give a four-way splitting. This comparatively fine structure then allows us to construct a number of different curvature operators:\ some of these generalize and refine operators (such as the Jacobi endomorphism) already known for ODEs, whereas others are new and so provide potential for further investigation. 

As demonstrations of the utility of our approach we present three applications. The first is the development of a hypothesis concerning separability in coordinates of a system of connection PDEs. In a 1996 paper Cantrijn et al~\cite{Canetal96} identified necessary and sufficient curvature conditions for a system of second order ordinary differential equations to be separable in co-ordinates. As a first step towards generalising this result to PDEs we provide a set of necessary conditions, akin to those in~\cite{Canetal96}, for a system of connection PDEs to be separable in coordinates. Secondly, we show that for harmonic maps our curvature constructions are exactly those we expect in this case so that our generalisation is fit for purpose in the same way as the ODE constructions generalise the Riemannian curvature for a linear connection. Thirdly, we present a concrete example which is not of harmonic map type; we shall carry this example into a future paper where we use our curvature constructions to study singularity formation (surface collapse) of congruences of solutions of PDEs of connection type.

Although, as we have said, our results apply directly to equations of connection type, the range of potential applications is much wider, as connection equations may describe important classes of solutions of more complicated equations. Indeed, second order variational equations (arising from an Euler--Lagrange operator) are equivalent to a parametrized family of equations of connection type, as described explicitly in~\cite{RoSa14}; examples include harmonic map equations and many field equations used in theoretical physics.

To assist the reader we shall, in the body of this paper, recall the construction for ODEs before describing the geometric background used in the PDE case.


\section{Geometric structures for second order ODEs}

We start by recalling the various constructions used for time-dependent second order ODEs which we aim to generalize~\cite{Kra97}; these are equations of the form given in~\eqref{Econnpde} but with $n=1$, so that they may be written as
\begin{equation*}
\frac{d^2 y^\sigma}{dt^2} = F^\sigma \biggl( t, y^\nu, \frac{dy^\nu}{dt} \biggr) \qquad 1 \le \sigma, \nu \le m.
\end{equation*}
Let $Y$ be a differentiable manifold of dimension $m+1$, such that $\pi : Y \to \R$ is a fibred manifold (so that $\pi$ is a surjective submersion, or equivalently admits local sections). The first jet manifold $J^1\pi$ contains equivalence classes $j^1_x \phi$ of local sections $\phi : I \to Y$, where $I \subset \R$ is an nonempty open subset; two local sections are equivalent if they have the same value and first derivatives at $x \in \R$. The second jet manifold $J^2\pi$ is defined in a similar way. The maps $\pi_1 : J^1 \pi \to \R$ and $\pi_{1,0} : J^1\pi \to Y$, defined by $\pi_1(j^1_x\phi) = x$, $\pi_{1,0}(j^1_x\phi) = \phi(x)$, are the first order source and target projections; the corresponding second order maps are $\pi_2 : J^2 \pi \to \R$ and $\pi_{2,0} : J^2\pi \to Y$. Finally the order reduction map $\pi_{2,1} : J^2 Y \to J^1 Y$ is defined by $\pi_{2,1}(j^2_x\phi) = j^1_x\phi$.

We let $t$ denote the identity coordinate on $\R$, and take fibred coordinates $(t, y^\sigma)$ on $Y$. The additional coordinates on the jet manifolds are defined by
\begin{equation*}
\ydot^\sigma(j^1_x\phi) = \ydot^\sigma(j^2_x\phi) = \eval{\frac{d\phi^\sigma}{dt}}{x} \, , \qquad 
\yddot^\sigma(j^2_x\phi) = \eval{\frac{d^2\phi^\sigma}{dt^2}}{x} \, ;
\end{equation*}
thus the coordinates on $J^1\pi$ are $(t, y^\sigma, \ydot^\sigma)$ and those on $J^2\pi$ are $(t, y^\sigma, \ydot^\sigma, \yddot^\sigma)$.

Every local section $\phi$ of $\pi$ gives rise to local sections $j^1\phi$ of $\pi_1$ and $j^2\phi$ of $\pi_2$, both with the same domain as $\phi$, by setting
\begin{equation*}
j^1\phi(x) = j^1_x\phi \, , \qquad j^2\phi(x) = j^2_x\phi \qquad (x \in I) \, ;
\end{equation*}
these new sections are called prolongations of $\phi$.

It is not, however, the case that every local section of $\pi_1$ or $\pi_2$ is the prolongation of a local section of $\pi$. We say that a differential $1$-form $\theta$ on $J^1\pi$ or $J^2\pi$ is a contact form if the pullback $(j^1\phi)^* \theta$ or $(j^2\phi)^* \theta$ by any prolonged local section is necessarily zero. A local basis for the contact $1$-forms on $J^1\pi$ is given by $(\omega^\sigma)$ where $\omega^\sigma = dy^\sigma - \ydot^\sigma dt$, and a local basis for the contact $1$-forms on $J^2\pi$ is given by $(\omega^\sigma, \omdot^\sigma)$ where in addition $\omdot^\sigma = d\ydot^\sigma - \yddot^\sigma dt$. It is sometimes convenient to take $(dt, \omega^\sigma, d\ydot^\sigma)$ instead of $(dt, dy^\sigma, d\ydot^\sigma)$ as a local basis for all the $1$-forms on $J^1\pi$, and similarly to take $(dt, \omega^\sigma, \omdot^\sigma, d\yddot^\sigma)$ as a local basis for all the $1$-forms on $J^2\pi$. 

By identifying a local section $\phi$ with a curve in $Y$, we may regard $J^1\pi$ as a submanifold of $TY$; the jet $j^1_x\phi$ corresponds to the tangent vector
\begin{equation*}
\biggl( \vf{t} + \frac{d\phi^\sigma}{dt} \vf{y^\sigma} \biggr)_{\phi(x)} \in T_{\phi(x)}Y
\end{equation*}
so that $\pi_{1,0} : J^1\pi \to Y$ becomes an affine sub-bundle of $TY \to Y$, modelled on the bundle of vertical tangent vectors $V\pi \to Y$. In the second order case we identify a prolonged local section $j^1\phi$ with a curve in $J^1\pi$, and so regard $J^2\pi$ as a submanifold of $TJ^1\pi$; the jet $j^2_x\phi$ corresponds to the tangent vector
\begin{equation*}
\biggl( \vf{t} + \ydot^\sigma \vf{y^\sigma} + \frac{d^2\phi^\sigma}{dt^2} \vf{\ydot^\sigma} \biggr)_{j^1_x\phi} \in T_{j^1_x\phi}J^1\pi
\end{equation*}
so that $\pi_{2,1} : J^2\pi \to J^1\pi$ becomes an affine sub-bundle of $TJ^1\pi \to J^1\pi$, modelled on the bundle of vertical tangent vectors $V\pi_{1,0} \to J^1\pi$.

We now construct an endomorphism of $TJ^1\pi$, a vector-valued $1$-form $S_1$ on $J^1\pi$ called the vertical endomorphism. Take a point $j^1_x\phi \in J^1\pi$, and consider the linear map
\begin{equation*}
v_{j^1_x\phi} = \id - T\phi \circ T\pi : T_{\phi(x)}Y \to T_{\phi(x)}Y \, .
\end{equation*}
Following this map with $T\pi$ gives zero, showing that it takes its values in the vertical subspace $V_{\phi(x)}\pi \subset T_{\phi(x)}Y$. But $V_{j^1_x\phi}\pi_{1,0}$ is the tangent space to an affine space, so there is a canonical isomorphism $l_{j^1_x\phi} : V_{\phi(x)}\pi \to V_{j^1_x\phi}\pi_{1,0}$. The composite linear map
\begin{equation*}
T_{j^1_x\phi}J^1\pi \xrightarrow{T\pi_{1,0}} T_{\phi(x)}Y \xrightarrow{v_{j^1_x\phi}} V_{\phi(x)}\pi \xrightarrow{l_{j^1_x\phi}} V_{j^1_x\phi}\pi_{1,0}
\subset T_{j^1_x\phi}J^1\pi
\end{equation*}
then describes the action of $S_1$ on the fibre of $TJ^1\pi$ at $j^1_x\phi$; in coordinates
\begin{equation*}
S_1 = \omega^\sigma \otimes \vf{\ydot^\sigma} \, .
\end{equation*}
There is a related endomorphism of $TJ^2\pi$, now a vector-valued $1$-form $S_2$ on $J^2\pi$ and again called the vertical endomorphism, but we cannot construct it in the same way from the affine bundle structure of $\pi_{2,1} : J^2\pi \to J^1\pi$ because the first part of the composite linear map would be 
\begin{equation*}
T_{j^2_x\phi}J^2\pi \xrightarrow{T\pi_{2,1}} T_{j^1_x\phi}J^1\pi \xrightarrow{v_{j^2_x\phi}} V_{j^1_x\phi}\pi_1
\end{equation*}
where $v_{j^2_x\phi} = \id - Tj^1\phi \circ T\pi_1$, whereas the final map obtained from the affine structure would be
\begin{equation*}
V_{j^1_x\phi}\pi_{1,0} \xrightarrow{l_{j^2_x\phi}} V_{j^2_x\phi}\pi_{2,1} \subset T_{j^2_x\phi}J^2\pi
\end{equation*}
so that the two parts of the proposed composite do not fit together. Instead we take a more general approach to obtain a more appropriate final map. Consider a point $j^2_x\phi \in J^2\pi$ and a tangent vector $\zeta \in V_{j^1_x\phi}\pi_1$, and let $\phi_t$ be a one-parameter family of local sections satisfying the conditions that $j^2_x\phi_0 = j^2_x\phi$ and $d(j^1_x\phi_t) / dt|_{t=0} = \zeta$. Such a one-parameter family always exists, and the tangent vector
\begin{equation*}
\zeta^{\vl} = \eval{\frac{d(j^1_x\chi_t)}{dt}}{t=0} \, ,
\end{equation*}
where the new family of local sections $\chi_t$ is given by $\chi_t(x) = \phi_{tx}(x)$, is independent of the choice of $\phi_t$ and satisfies $T\pi_{2,0}(\zeta^{\vl}) = 0$. This construction gives a linear isomorphism $V_{j^1_x\phi}\pi_1 \to V_{j^2_x\phi}\pi_{2,1}$, and may be used to complete the composite map
\begin{equation*}
T_{j^2_x\phi}J^2\pi \xrightarrow{T\pi_{2,1}} T_{j^1_x\phi}J^1\pi \xrightarrow{v_{j^2_x\phi}} V_{j^1_x\phi}\pi_1
\xrightarrow{\zeta \mapsto \zeta^{\vl}} V_{j^2_x\phi}\pi_{2,1} \subset T_{j^2_x\phi}J^2\pi \, .
\end{equation*}
In coordinates
\begin{equation*}
S_2 = \omega^\sigma \otimes \vf{\ydot^\sigma} + 2 \omdot^\sigma \otimes \vf{\yddot^\sigma} \, ,
\end{equation*}
and we see that $T\pi_{2,1} \circ S_2 = S_1 \circ T\pi_{2,1}$.

All the objects described so far are general, and arise on any jet bundle over $\R$. We now consider a particular choice of second order ODE.

A vector field $\Gamma$ on $J^1\pi$ is called a second order vector field, or sometimes a second order differential equation field (\textsc{sode}) or a semispray, if its contractions with the time $1$-form and the vertical endomorphism satisfy
\begin{equation*}
i_\Gamma dt = 1 \, , \qquad i_\Gamma S_1 = 0 \, .
\end{equation*}
These two conditions force the coordinate representation of $\Gamma$ to be
\begin{equation*}
\Gamma = \vf{t} + \ydot^\sigma \vf{y^\sigma} + F^\sigma \vf{\ydot^\sigma} \, ,
\end{equation*}
so we also see that $i_\Gamma \omega^\sigma = 0$ for any basis contact form $\omega^\sigma$.

We now use the vector field $\Gamma$ to write the tangent bundle $TJ^1\pi$ as a direct sum
\begin{equation*}
TJ^1\pi = \dist_- \oplus \Span{\{\Gamma\}} \oplus V\pi_{1,0} = \dist_- \oplus \dist_0 \oplus \dist_+
\end{equation*}
where $\dist_+$ is the vertical sub-bundle with respect to the projection $\pi_{1,0}$, $\dist_0$ is the sub-bundle spanned by the second order vector field itself, and $\dist_-$ is a `horizontal sub-bundle' which will now be defined.

Consider the Lie derivative vector-valued $1$-form $\Lie_\Gamma S_1$, with coordinate expression
\begin{equation*}
\Lie_\Gamma S_1 = \biggl( d\ydot^\sigma - F^\sigma dt - \pd{F^\sigma}{\ydot^\nu} \omega^\nu \biggr) \otimes \vf{\ydot^\sigma} 
- \omega^\sigma \otimes \vf{y^\sigma} \, ,
\end{equation*}
regarded as a family of linear operators on the fibres of $TJ^1\pi$:\ we see that there are three eigendistributions, corresponding to the eigenvalues zero and $\pm 1$. The $+1$ eigendistribution is just $\dist_+$ with rank $m$, and the zero eigendistribution is $\dist_0$ and so has rank one. The $-1$ eigendistribution $H = \dist_-$ again has rank $m$ and is spanned by the vector fields
\begin{equation*}
H_\sigma = \vf{y^\sigma} - \Gamma^\nu_\sigma \vf{\ydot^\nu} \, , \qquad \Gamma^\nu_\sigma = - \tfrac{1}{2} \pd{F^\nu}{\ydot^\sigma} \, .
\end{equation*}

There is an alternative approach to the decomposition of $TJ^1\pi$, using the second order vertical endomorphism $S_2$~\cite{Sau97}. To use this, we note from the coordinate formula for the vector field $\Gamma$ that the second order condition is precisely the constraint needed for $\Gamma$ to take its values in the affine submanifold $J^2\pi \subset TJ^1\pi$. We may therefore regard $\Gamma$ as a section of $\pi_{2,1} : J^2\pi \to J^1\pi$. The image $\Gamma(J^1\pi) \subset J^2\pi$ may be regarded as a submanifold defined locally by the equations $\yddot^\sigma = F^\sigma$, and so the annihilator $T^\circ (\Gamma(J^1\pi)) \subset T^*_{\Gamma(J^1\pi)} J^2\pi$ of its tangent space is a codistribution along the image, spanned locally by the $1$-forms
\begin{equation*}
d\yddot^\sigma - dF^\sigma \, .
\end{equation*}
Operating on this codistribution with the second order vertical endomorphism $S_2$ gives a new codistribution $S_2 \bigl( T^\circ (\Gamma(J^1\pi)) \bigr)$ along the image, now spanned locally by the $1$-forms
\begin{equation*}
S_2(d\yddot^\sigma - dF^\sigma) = 2\omdot^\sigma - \pd{F^\sigma}{\ydot^\nu} \omega^\nu = 2(\omdot^\sigma + \Gamma^\sigma_\nu \omega^\nu) \, ,
\end{equation*}
and finally we may pull this back to $J^1\pi$ by the section $\Gamma$ to give the codistribution $\Gamma^* S_2 \bigl( T^\circ (\Gamma(J^1\pi)) \bigr)$ on $J^1\pi$, spanned locally by the $1$-forms
\begin{equation*}
\psi^\sigma = \ombar^\sigma + \Gamma^\sigma_\nu \omega^\nu
\end{equation*}
where $\ombar^\sigma = d\ydot^\sigma - F^\sigma dt$. We see that
\begin{align*}
i_\Gamma dt & = 1 & i_\Gamma \omega^\sigma & = 0 & i_\Gamma \psi^\sigma & = 0 \\
i_{H_\nu} dt & = 0 & i_{H_\nu} \omega^\sigma & = \delta^\sigma_\nu & i_{H_\nu} \psi^\sigma & = 0 \\
i_{V_\nu} dt & = 0 & i_{V_\nu} \omega^\sigma & = 0 & i_{V_\nu} \psi^\sigma & = \delta^\sigma_\nu 
\end{align*}
showing that this codistribution, together with the two codistributions spanned by $dt$ and by the contact forms $\omega^\sigma$, gives a decomposition of $T^* J^1\pi$ dual to the decomposition of $TJ^1\pi$ by the three eigendistributions of $\Lie_\Gamma S_1$. 

To conclude this section we consider Lie bracket relations between basis elements of the three eigendistributions. Writing $V_\sigma = \p / \p \ydot^\sigma$ we see that
\begin{subequations}
\begin{align}
[\Gamma, H_\sigma] & = \Gamma^\nu_\sigma H_\nu + \Phi^\nu_\sigma V_\nu \label{Jacobi} \\
[\Gamma, V_\nu] & = - H_\nu + \Gamma^\sigma_\nu V_\sigma \\
[H_\sigma, H_\nu] & = \curv^\rho_{\nu\sigma} V_\rho
\end{align}
\end{subequations}
where the symbols $\curv^\rho_{\nu\sigma} = H_\nu (\Gamma^\rho_\sigma) - H_\sigma (\Gamma^\rho_\nu)$ are the components of a vector-valued $2$-form on $J^1\pi$ representing the curvature of the distribution $\dist_-$, and the symbols $\Phi^\nu_\sigma = \Gamma^\rho_\sigma \Gamma^\nu_\rho - \Gamma(\Gamma^\nu_\sigma) - H_\sigma( F^\nu)$ are the components of a vector-valued $1$-form $\Phi$ on $J^1\pi$ known as the Jacobi endomorphism and representing certain curvature components of $\dist_-\oplus \dist_0$.


\section{Geometric structures for second order PDEs of connection type}

We now describe the general structure of jet manifolds used where there are several independent variables~\cite{Sau89}.

Let $\pi : Y \to X$ be a fibred manifold where $\dim X = n$ and $\dim Y = m+n$. The first and second jet manifolds $J^1\pi$ and $J^2\pi$ are defined in the same way as before, as are the source, target and order reduction maps. We now let $x^i$ be coordinates on $X$, and we take fibred coordinates $(x^i, y^\sigma)$ on $Y$. The additional coordinates on the jet manifolds are now defined by
\begin{equation*}
y^\sigma_i(j^1_x\phi) = y^\sigma_i(j^2_x\phi) = \eval{\pd{\phi^\sigma}{x^i}}{x} \, , \qquad 
y^\sigma_{ij}(j^2_x\phi) = \eval{\pdb{\phi^\sigma}{x^i}{x^j}}{x} \, ;
\end{equation*}
thus the coordinates on $J^1\pi$ are $(x^i, y^\sigma, y^\sigma_i)$ and those on $J^2\pi$ are $(x^i, y^\sigma, y^\sigma_i, y^\sigma_{ij})$. Note that, in view of the symmetry of second partial derivatives, $y^\sigma_{ji} = y^\sigma_{ij}$:\ this has consequences for our use of the summation convention, and so we introduce the symbol $n(ij)$ giving the number of distinct values represented by the indices $i$ and $j$.

The prolongations of local sections are defined in the same way as before, as are contact forms, although we now extend the use the of latter description to $r$-forms where $1 \le r \le n$. A local basis for the contact $1$-forms on $J^1\pi$ is given by $(\omega^\sigma)$ where now $\omega^\sigma = dy^\sigma - y^\sigma_i dx^i$, and a local basis for the contact $1$-forms on $J^2\pi$ is given by $(\omega^\sigma, \omega^\sigma_i)$ where in addition $\omega^\sigma_i = dy^\sigma_i - y^\sigma_{ij} dx^j$. We often take $(dx^i, \omega^\sigma, dy^\sigma_i)$ instead of $(dx^i, dy^\sigma, dy^\sigma_i)$ as a local basis for all the $1$-forms on $J^1\pi$, and similarly take $(dx^i, \omega^\sigma, \omega^\sigma_i, dy^\sigma_{ij})$, with a suitable understanding about the symmetric pair of indices $(ij)$, as a local basis for all the $1$-forms on $J^2\pi$.

The contact forms on $J^1\pi$ (or $J^2\pi$) define distributions on these two manifolds, called the contact (or Cartan) distributions, and also known as the contact structures (of orders $1$ and $2$) on $\pi$. These distributions are
\begin{align*}
\cart_{\pi_1} & = \annih \{ \omega^\sigma \} = \Span \biggl\{ \vf{x^i} + y^\sigma_i \vf{y^\sigma}, \vf{y^\nu_k} \bigg\} \\
\cart_{\pi_2} & = \annih \{ \omega^\sigma, \omega^\sigma_i \} 
= \Span \biggl\{ \vf{x^i} + y^\sigma_i \vf{y^\sigma} + y^\sigma_{ij} \vf{y^\sigma_j}, \vf{y^\nu_{kl}} \bigg\} \, .
\end{align*}

It is again possible to construct vertical endomorphisms~\cite{Sau87}, but there is an added complication because, with $n > 1$, the ranks of the bundles of vertical vectors are no longer equal, and so they cannot be isomorphic:
\begin{equation*}
\rank V\pi = m \, , \qquad \rank V\pi_{1,0} = mn \, , \qquad \rank V\pi_{2,1} = \tfrac{1}{2}mn(n+1) \, .
\end{equation*}
We deal with this problem by selecting a nonvanishing $1$-form $\varphi$ on $X$ (we assume the topology of $X$ is such that this is possible) and in the first order case we use this to select a distinguished sub-bundle
\begin{equation*}
\pi^* (\Span{\varphi}) \otimes V\pi \subset \pi^* (T^* X) \otimes V\pi \cong V\pi_{1,0}
\end{equation*}
of rank $m$, where the isomorphism above arises from the structure of $J^1\pi \to Y$ as an affine bundle, now modelled on the vector bundle $\pi^*(T^* X) \otimes V\pi \to Y$. We may therefore define the $\varphi$-vertical endomorphism $S_1^\varphi$ by constructing, at each point $j^1_x\phi \in J^1\pi$, the composite linear map
\begin{equation*}
T_{j^1_x\phi}J^1\pi \xrightarrow{T\pi_{1,0}} T_{\phi(x)}Y \xrightarrow{v_{j^1_x\phi}} V_{\phi(x)}\pi 
\xrightarrow{l_{j^1_x\phi, \varphi}} V_{j^1_x\phi}\pi_{1,0} \subset T_{j^1_x\phi}J^1\pi
\end{equation*}
in the same way as before, where with $n > 1$ the final map
\begin{equation*}
l_{j^1_x\phi, \varphi} : V_{\phi(x)}\pi \to V_{j^1_x\phi}\pi_{1,0}
\end{equation*}
depends on $\varphi$ and is no longer surjective. In coordinates we have
\begin{equation*}
S_1^\varphi = \varphi_i \omega^\sigma \otimes \vf{y^\sigma_i} \, , \qquad \varphi = \varphi_i dx^i \, .
\end{equation*}

In the second order case we make the additional assumption that $\varphi$ is closed, so that locally there is some function $f$ such that $\varphi = df$. We again consider a point $j^2_x\phi \in J^2\pi$ and a tangent vector $\zeta \in V_{j^1_x\phi}\pi_1$, letting $\phi_t$ be a one-parameter family of local sections satisfying the conditions that $j^2_x\phi_0 = j^2_x\phi$ and $d(j^1_x\phi_t) / dt|_{t=0} = \zeta$. The difference is that we now define the modified one-parameter family $\chi_t$ by $\chi_t(x) = \phi_{tf(x)}(x)$, and then the tangent vector
\begin{equation*}
\zeta^{\vl\{\varphi\}} = \eval{\frac{d(j^1_x\chi_t)}{dt}}{t=0} \, ,
\end{equation*}
is independent both of the choice of $\phi_t$ and of the choice of local function $f$; again it satisfies $T\pi_{2,0}(\zeta^{\vl\{\varphi\}}) = 0$. This construction gives a linear injection $V_{j^1_x\phi}\pi_1 \to V_{j^2_x\phi}\pi_{2,1}$, and may be used to create the composite map
\begin{equation*}
T_{j^2_x\phi}J^2\pi \xrightarrow{T\pi_{2,1}} T_{j^1_x\phi}J^1\pi \xrightarrow{v_{j^2_x\phi}} V_{j^1_x\phi}\pi_1
\xrightarrow{\zeta \mapsto \zeta^{\vl\{\varphi\}}} V_{j^2_x\phi}\pi_{2,1} \subset T_{j^2_x\phi}J^2\pi \, .
\end{equation*}
The coordinate formula for this map $S_2^\varphi$ is more complicated:\ it is
\begin{equation*}
S_2^\varphi = \varphi_i \omega^\sigma \otimes \vf{y^\sigma_i} + \pd{\varphi_i}{x^j} \omega^\sigma \otimes \vf{y^\sigma_{ij}} 
+ \frac{2}{n(ij)} \varphi_i \omega^\sigma_j \otimes \vf{y^\sigma_{ij}} \, .
\end{equation*}

The vertical endomorphisms $S_1^\varphi$ and $S_2^\varphi$ are both vector valued $1$-forms, and we make use of several other vector-valued forms in our discussion, so we shall need to consider a bracket operation on these objects. This will be the \FN\ bracket of vector valued forms~\cite{FrNi56}, defined by
\begin{align*}
\fnb{\omega \otimes U, \eta \otimes W} & = \omega \wedge \eta \otimes [U,W] + \omega \wedge \Lie_U \eta \otimes W 
- \Lie_W \omega \wedge \eta \otimes U \\
& \quad + (-1)^{\deg \omega} ( d\omega \wedge i_U \eta \otimes W + i_W \omega \wedge d\eta \otimes U )
\end{align*}
where $U$, $W$ are vector fields and $\omega$, $\eta$ are forms (of any degree). Indeed if $\omega$ is a $0$-form (function) then the \FN\ bracket is simply the Lie derivative of $\eta \otimes W$ by the vector field $\omega U$.

We turn finally to second order PDEs of the type we are going to consider. A general system of second order PDEs may be regarded geometrically as a submanifold of the second jet manifold $J^2\pi$; we shall consider those systems where the submanifold is the image of a section $\Gamma$ of the affine bundle $\pi_{2,1} : J^2\pi \to J^1\pi$. Such systems may be called second order connections, or semispray connections, and in general are overdetermined systems of PDEs. Defining local functions $F^\sigma_{ij} = y^\sigma_{ij} \circ \Gamma$ on $J^1\pi$, the system of PDEs is given in coordinates by
\begin{equation*}
\pdb{\phi^\sigma}{x^i}{x^j} = F^\sigma_{ij} \biggl( x^k, \phi^\nu, \pd{\phi^\nu}{x^k} \biggr) \, .
\end{equation*}
We may also represent this system by a vector-valued $1$-form on $J^1\pi$, in the same way as a system of second order ODEs may be represented by a second order vector field. Given any point $j^1_x\phi \in J^1\pi$ and a tangent vector $\zeta \in T_x X$, define the tangent vector $\pdeform_{j^1_x\phi}(\zeta) \in T_{j^1_x\phi}J^1\pi$ by
\begin{equation*}
\pdeform_{j^1_x\phi}(\zeta) = \eval{(Tj^1\phi)(\zeta)}{\Gamma(j^1_x\phi)} \, .
\end{equation*}
In this formula the tangent map $Tj^1\phi$ depends on the derivatives of the prolonged section $j^1\phi$ at $x$, and hence on the second derivatives of the original section $\phi$ at $x$; these second derivatives are given by the point $\Gamma(j^1_x\phi) \in J^2\pi$. This procedure defines a linear map $T_x X \to T_{j^1_x\phi}J^1\pi$, so that the projection $T\pi_1 : TJ^1\pi \to TX$ followed by the map $\zeta \mapsto \pdeform_{j^1_x\phi}(\zeta)$ defines a vector-valued $1$-form $\pdeform$ on $J^1\pi$ given in coordinates by
\begin{equation*}
\pdeform = dx^i \otimes \biggl( \vf{x^i} + y^\sigma_i \vf{y^\sigma} + F^\sigma_{ij} \vf{y^\sigma_j} \biggr) = dx^i \otimes \Gamma_i
\end{equation*}
where $\Gamma_i$ are vector fields on $J^1\pi$; note that, by construction, the functions $F^\sigma_{ij}$ are symmetric in the indices $i$ and $j$. The vector fields $\Gamma_i$ are local generators of a distribution $D_\Gamma$ on $J^1\pi$ which is a horizontal sub-bundle of the contact distribution $\cart_{\pi_1}$. We also write $\Gamma_v$ for the vector field on $J^1\pi$ corresponding to the vector field $v$ on $M$, so that $\Gamma_v = v^i \Gamma_i = i_v \pdeform$.


\section{Tangent bundle decompositions for systems of PDEs}

We have seen that a system of second order ordinary differential equations, represented by a second order vector field $\Gamma$ on $J^1\pi$, induces a splitting of the tangent bundle $TJ^1\pi$ into the three eigenspaces of the deformation of the vertical endomorphism $S$ along the vector field 
$\Gamma$. One might wonder whether a similar property holds for PDEs in the form of second order connections. In fact such splittings are possible but, as we shall show, the process is not completely straightforward and the details of the result are unexpected.

The first difference is related with the vertical endomorphism, which now depends upon a choice of a non-vanishing $1$-form $\varphi$ on $X$, and so gives an isomorphism of $V\pi$ with an $m$-dimensional sub-bundle $S_1^\varphi(V\pi_1)$ of $V\pi_{1,0}$. The second difference, giving rise to a major complication, comes from the fact that instead of a single vector field (corresponding to a distribution $\dist_0$ of rank one) we now have on $J^1\pi$ a distribution $D_{\Gamma}$ of rank $n$, so that the connection is represented by a vector valued $1$-form $\pdeform$. (The notational distinction, using $\dist_0$ and $D_{\Gamma}$, is, as we shall see, significant.) The deformation becomes a \FN\ bracket $\fnb{\pdeform, S^\varphi_1}$, a vector valued $2$-form on $J^1\pi$, given in coordinates by
\begin{equation*}
\fnb{\pdeform, S_1^\varphi} = - \, \varphi_i dx^i \wedge \omega^\nu \otimes \vf{y^\nu} 
+ dx^i \wedge \biggl( \varphi_k dy^\nu_i + \pd{\varphi_k}{x^i} \omega^\nu - \varphi_j \pd{F^\nu_{ik}}{y^\sigma_j} \omega^\sigma \biggr)
\otimes \vf{y^\nu_k} \, .
\end{equation*}
In order to obtain eigenspaces, we need a vector-valued $1$-form rather than a vector-valued $2$-form, and so we need to take directional slices.
\begin{lemma}
Let $v$ be a vector field on $X$, and let $\Gamma_v$ be the vector field on $J^1\pi$ defined by $\Gamma_v = i_v \pdeform$. Then
\[
i_{\Gamma_v} \fnb{\pdeform, S_1^\varphi} = \Lie_{\Gamma_v} S_1^\varphi \, .
\]
\end{lemma}
\begin{proof}
Using the coordinate expressions for $\pdeform$ and $S_1^\varphi$ we see first that
\begin{equation*}
\Gamma_v = v^i \biggl( \vf{x^i} + y^\sigma_i \vf{y^\sigma} + F^\sigma_{ij} \vf{y^\sigma_j} \biggr) \, , \qquad v = v^i \vf{x^i}
\end{equation*}
and then that
\begin{align*}
\Lie_{\Gamma_v} S_1^\varphi & = \Gamma_v(\varphi_k) \omega^\nu \otimes \vf{y^\nu_k} + \varphi_k (\Lie_{\Gamma_v} \omega^\nu) \otimes \vf{y^\nu_k}
+ \varphi_k \omega^\nu \biggl[ \Gamma_v, \vf{y^\nu_k} \biggr] \\
& = v^i \pd{\varphi_k}{x^i} \omega^\nu \otimes \vf{y^\nu_k} + \varphi_k v^i (dy^\nu_i - F^\nu_{ij} dx^j) \otimes \vf{y^\nu_k} \\
& \qquad - v^k \varphi_k \omega^\nu \otimes \vf{y^\nu} - v^i \varphi_k \pd{F^\sigma_{ij}}{y^\nu_k} \omega^\nu \otimes \vf{y^\sigma_j} \\
& = i_{\Gamma_v} \fnb{\pdeform, S_1^\varphi} \, .
\end{align*}
\end{proof}

We now wish to consider possible eigenvectors; it is clear that we need to restrict attention to the case where $i_v\varphi \ne 0$. We note also that if $U$ is an eigenvector of $\Lie_{\Gamma_v} S_1^\varphi$ then it is also an eigenvector of $\Lie_{\Gamma_{hv}} S_1^\varphi$ with an eigenfunction scaled by $h$; in other words, the directional slice corresponds, not to a specific vector field $v$ on $X$, but to a distribution of rank one, locally spanned by a vector field $v$ such that $i_v \varphi \ne 0$.

Consider possible eigenvectors
\begin{equation*}
U = U^h \vf{x^h} + U^\nu \vf{y^\nu} + U^\nu_h \vf{y^\nu_h} \, ;
\end{equation*}
we have, after some rearranging,
\begin{align*}
i_U \Lie_{\Gamma_v} S_1^\varphi & = - \, v^i \varphi_i (U^\nu - U^h y^\nu_h) \vf{y^\nu}
+ v^i \Bigl( \varphi_k (U^\nu_i - F^\nu_{ij} U^j) \, + \\
& \qquad + \pd{\varphi_k}{x^i} (U^\nu - U^h y^\nu_h) 
- \varphi_j \pd{F^\nu_{ik}}{y^\sigma_j} (U^\sigma - U^h y^\sigma_h) \Bigr) \vf{y^\nu_k} \, .
\end{align*}
We see immediately that any $\Gamma_w$ is an eigenvector with eigenvalue zero; we also see that any vertical vector field satisfying $v^i U^\nu_i = 0$ is an eigenvector with eigenvalue zero. The zero eigendistribution $\dist_0$ therefore satisfies $D_\Gamma \subset \dist_0$ and has rank at least $n + m(n-1)$.

We next consider the possibility of eigenvalues having the same sign as $i_v\varphi$. This can arise only with a vertical vector field; we see that any such vector field must be of the form
\begin{equation*}
U = U^\nu \varphi_k \vf{y^\nu_k} = S_1^\varphi(\bar{U}) \, , \qquad \bar{U} = U^\nu \vf{y^\nu}
\end{equation*}
so that
\begin{equation*}
i_U \Lie_{\Gamma_v} S_1^\varphi = (v^i \varphi_k) U^\nu \varphi_i \vf{y^\nu_k} = (v^i \varphi_i) \varphi_k U^\nu \vf{y^\nu_k} = (v^i \varphi_i) U
\end{equation*}
with an eigenvalue of $v^i \varphi_i = i_v\varphi$, so that the scaled vector field $hv$ with $h = (i_v\varphi)^{-1}$ would give rise to the same eigendistribution, but with an eigenvalue of $+1$. This eigendistribution $\dist_+$ has rank $m$.

Finally we consider the possibility of eigenvalues having the opposite sign. We see that by taking
\begin{equation*}
U = U^\nu \vf{y^\nu} + U^\nu_k \vf{y^\nu_k}
\end{equation*}
we obtain
\begin{equation*}
i_U \Lie_{\Gamma_v} S_1^\varphi 
= - v^i \varphi_i U^\nu \vf{y^\nu} 
+ v^i \biggl( \varphi_k U^\nu_i + \pd{\varphi_k}{x^i} U^\nu - \varphi_j \pd{F^\nu_{ik}}{y^\sigma_j} U^\sigma \biggr) \vf{y^\nu_k}
\end{equation*}
so that the eigenvalue must be $-v^i\varphi_i$, and therefore that we must have
\begin{equation*}
v^i \varphi_i U^\nu_k = - v^i \biggl( \varphi_k U^\nu_i + \pd{\varphi_k}{x^i} U^\nu - \varphi_j \pd{F^\nu_{ik}}{y^\sigma_j} U^\sigma \biggr)
\end{equation*}
so that
\begin{equation}
\label{Ehorsoln}
v^i(\varphi_i U^\nu_k + \varphi_k U^\nu_i) 
= v^i U^\sigma \biggl( \varphi_j \pd{F^\nu_{ik}}{y^\sigma_j} - \delta^\nu_\sigma \pd{\varphi_k}{x^i} \biggr) \, .
\end{equation}
We see that the terms in the bracket on the left-hand side are symmetric in the indices $i$ and $k$, and by construction the functions $F^\nu_{ik}$ are also symmetric in the indices $i$ and $k$. We therefore see that, if equation~\eqref{Ehorsoln} is to have solutions for all vector fields $v$ satisfying $i_v\varphi \ne 0$ then the term $\p \varphi_k / \p x^i$ must also be symmetric in the indices $i$ and $k$, so that the $1$-form $\varphi$ needs to be closed. We shall therefore make that assumption.

In order to see that~\eqref{Ehorsoln} does indeed have solutions, suppose that we have chosen a normalized local vector field $v$ satisfying $i_v \varphi = 1$. In order to obtain explicit formulas for the eigenspaces it is convenient to use adapted coordinates on $X$ corresponding to the pair $(\varphi,v)$, namely such that
\begin{equation}
\varphi = dx^1 \, , \quad  i_v \varphi = v^1 = 1 \, .
\end{equation}
Such charts are guaranteed by the Frobenius theorem, and indeed $\varphi$ annihilates an integrable distribution of rank $n-1$ on $X$. For convenience we shall use indices $p$, $q$ satisfying $2 \le p,q \le n$, while continuing with $1 \le i, j, \ldots \le n$. In these adapted coordinates we obtain, after some straightforward calculations,
\begin{equation*}
2U^\nu_1 = \biggl( U^\sigma v^i \, \pd{F^\nu_{i1}}{y^\sigma_1} - U^\nu_q v^q \biggr) \, , \qquad
U^\nu_q = \biggl( U^\sigma v^i \pd{F^\nu_{iq}}{y^\sigma_1} \biggr)
\end{equation*}
and we see that the eigendistribution $\dist_-$ corresponding to the eigenvalue $-1$ has rank $m$ and is spanned locally by the vector fields
\begin{gather*}
H_\sigma = \vf{y^\sigma} + H^\nu_{\sigma k} \vf{y^\nu_k} \\
H^\nu_{\sigma 1} = \tfrac{1}{2} \biggl( \pd{F^\nu_{11}}{y^\sigma_1} - v^p v^q \pd{F^\nu_{pq}}{y^\sigma_1} \biggr) \, , \qquad
H^\nu_{\sigma q} = \pd{F^\nu_{1q}}{y^\sigma_1} + v^p \pd{F^\nu_{pq}}{y^\sigma_1} = v^k \pd{F^\nu_{kq}}{y^\sigma_1} \, ;
\end{gather*}
we also see that
\begin{equation*}
2v^k H_{\sigma k}^\nu =  \pd{F^\nu_{11}}{y^\sigma_1} - v^p v^q \pd{F^\nu_{pq}}{y^\sigma_1} + 2 v^p v^k \pd{F^\nu_{pk}}{y^\sigma_1}
= v^i v^k \pd{F^\nu_{ik}}{y^\sigma_1} \, .
\end{equation*}
In summary, therefore, we have proved the following result.
\begin{theorem}
Let $\Gamma: J^1\pi \to J^2\pi$ be a second order connection. Given a nonvanishing closed $1$-form $\varphi$ and a vector field $v$ on $X$ such that $i_v \varphi \ne 0$, the vector-valued $1$-form $\Lie_{\Gamma_v} S_1^\varphi$ has three eigenvalues $\lambda_1 = 0$, $\lambda_2 > 0$ and $\lambda_3 = -\lambda_2 < 0$, corresponding to eigendistributions $\dist_0$ of rank $mn + n - m$, $\dist_+$ of rank $m$ and $\dist_-$ of rank $m$, such that
\begin{equation*}
TJ^1\pi = \dist_- \oplus \dist_0 \oplus \dist_+ \, . \tag{A}
\end{equation*} \qed
\end{theorem}
We can, in fact, extend this decomposition by noting that
\begin{equation*}
D_\Gamma \subset \dist_0 \, , \quad \dist_+ \subset V\pi_{1,0} \, , \quad (\dist_0 \cap V\pi_{1,0}) \oplus \dist_+ = V\pi_{1,0}
\end{equation*}
so that we can define another splitting
\begin{equation*}
TJ^1\pi = \dist_- \oplus D_\Gamma \oplus V\pi_{1,0} \tag{B}
\end{equation*}
(Recall that, in the case of ordinary differential equations, the two splittings~(A) and~(B) were identical.) Considering both the splittings we get a refined fourfold splitting
\begin{equation*}
TJ^1\pi = \dist_- \oplus D_\Gamma \oplus (\dist_0 \cap V\pi_{1,0}) \oplus \dist_+ \tag{AB}
\end{equation*} 
into subspaces of dimensions $m$, $n$, $mn-m$, and $m$.


\section{Dual splittings of $T^* J^1\pi$}

Any splitting of the tangent bundle of a manifold into the direct sum of two or more distributions gives rise to a dual splitting of the cotangent bundle, by taking annihilators. In general we shall denote the annihilator of a distribution $\dist$ by the symbol $\dist^\circ$.

Consider first the splitting~(B) of $TJ^1\pi$,
\begin{equation*}
TJ^1\pi = \dist_- \oplus D_\Gamma \oplus V\pi_{1,0} \, ;
\end{equation*}
its dual is the splitting
\begin{equation*}
T^* J^1\pi = \cart_{\pi_1}^\circ \oplus \pi_1^* T^* X \oplus \force 
\end{equation*}
where $\cart_{\pi_1}^\circ$ is spanned by the contact forms $\omega^\sigma$, $\pi_1^* T^* X$ is spanned by the horizontal forms $dx^i$, and the `force bundle' $\force$ is spanned by forms
\begin{equation*}
\psi^\nu_k = dy^\nu_k - F^\nu_{ki} dx^i - H^\nu_{\sigma k} \omega^\sigma \, , \qquad 1 \le \nu \le m, \quad 1 \le k \le n.
\end{equation*}
Thus we may form a local basis of vector fields on $J^1\pi$ adapted to splitting~(B) and its dual local basis of $1$-forms,
\begin{equation*}
\biggl\{ H_\sigma, \Gamma_i, \vf{y^\nu_k} \biggr\} \, , \qquad \bigl\{ \omega^\sigma, dx^i, \psi^\nu_k \bigr\} \, .
\end{equation*}
The three projectors,
\begin{equation*}
\horproj : TJ^1\pi \to \dist_- \, , \quad \pdeform : TJ^1\pi \to D_\Gamma \, , \quad \vertproj : TJ^1\pi \to V\pi_{1,0} \, ,
\end{equation*}
vector-valued $1$-forms satisfying $\horproj + \pdeform + \vertproj = \identity$, are then expressed in terms of these local bases by
\begin{equation*}
\horproj = \omega^\sigma \otimes H_\sigma \, , \qquad \pdeform = dx^i \otimes \Gamma_i \, , \qquad \vertproj = \psi^\nu_k \otimes \vf{y^\nu_k} \, .
\end{equation*}
Refining to the four-fold splitting~(AB)
\begin{equation*}
TJ^1\pi = \dist_- \oplus D_\Gamma \oplus (\dist_0 \cap V\pi_{1,0}) \oplus \dist_+
\end{equation*}
is then simply a matter of decomposing the projector $\vertproj$ into two parts $\vertprojt + \vertproj_+$, corresponding to $V\pi_{1,0} = (\dist_0 \cap V\pi_{1,0}) \oplus \dist_+$. The four-fold dual splitting of $T^* J^1\pi$ is therefore
\begin{equation*}
T^* J^1\pi = \cart_{\pi_1}^\circ \oplus \pi_1^* T^* X \oplus (\dist_+^\circ \cap \force) \oplus \force_+ \, . 
\end{equation*}
Now using adapted coordinates where $\varphi = dx^1$, we may form a local basis of vector fields on $J^1\pi$ adapted to splitting~(AB) and its dual local basis of $1$-forms,
\begin{equation*}
\biggl\{ H_\sigma, \Gamma_i, W^p_\nu, \vf{y^\nu_1} \biggr\} \, , \qquad \bigl\{ \omega^\sigma, dx^i, \psi^\nu_p, v^k \psi^\nu_k \bigr\}
\end{equation*}
where
\begin{equation*}
W^p_\nu = \vf{y^\sigma_p} - v^p \vf{y^\nu_1} \, .
\end{equation*}


\section{Curvature operators}
\label{Scurv}

Any splitting of a tangent bundle into two complementary distributions, one of which is the vertical bundle over a fibration, gives rise to a curvature operator, a vector-valued $2$-form, obtained by taking the \FN\ bracket of one of the two projectors with itself. As the projectors sum to the identity, taking the \FN\ bracket of the other projector with itself, or taking the mixed bracket, will not make any essential difference. This arises when, for example, we take a straightforward two-way splitting of $TJ^1\pi$ as $D_\Gamma \oplus V\pi_1$:\ we obtain
\begin{equation*}
R^\Gamma = \fnb{\pdeform,\pdeform} = dx^i \wedge dx^j \otimes [\Gamma_i, \Gamma_j]
\end{equation*}
where
\begin{equation*}
[\Gamma_i, \Gamma_j] = \bigl( \Gamma_i(F^\nu_{jk}) - \Gamma_j(F^\nu_{ik}) \bigr) \vf{y^\nu_k}
\end{equation*}
takes its values in $V\pi_{1,0} \subset V\pi_1$ (and of course vanishes in the case of ODEs). But with a three-fold or a four-fold splitting, there are more games we can play. As well as the bracket of each projector with itself, we can also take mixed brackets; and in addition we can consider the components of each bracket under the splitting. We list many of these various curvatures in an Appendix; we shall mention just the most important of them here.

Possibly the most significant of these vector-valued $2$-forms is an operator $\Phi$ which we call the `Jacobi curvature' to emphasize its relationship with the Jacobi endomorphism from the ODE case appearing in Equation~\eqref{Jacobi}. We use splitting~(B), and take the component of $\fnb{\pdeform, \horproj}$ in $V\pi_{1,0}$ by writing
\begin{equation*}
\Phi = \vertproj \circ \fnb{\pdeform, \horproj} \, .
\end{equation*}
To see the expression of $\Phi$ in coordinates, we note that
\begin{equation*}
\fnb{\pdeform, \horproj} = \fnb{dx^i \otimes \Gamma_i, \omega^\sigma \otimes H_\sigma}
\end{equation*}
so that
\begin{equation*}
\Phi  = \Phi^\nu_{i \sigma j} dx^i \wedge \omega^\sigma \otimes \vf{y^\nu_j}
\end{equation*}
where
\begin{equation*}
\Phi^\nu_{i \sigma j} = H^\rho_{\sigma i} H^\nu_{\rho j} + \Gamma_i(H^\nu_{\sigma j}) - H_\sigma(F^\nu_{ij})
\end{equation*}
(we have a different sign convention for the functions $H^\nu_{\sigma j}$ from the ODE case, so that the middle term in this expression comes with a positive sign). We also write $\Phi_{i\sigma}$ for the vector field
\begin{equation*}
\Phi_{i\sigma} = \Phi^\nu_{i \sigma j} \vf{y^\nu_j} = [\Gamma_i, H_\sigma] + H^\nu_{\sigma i} H_\nu \, .
\end{equation*} 
We may also consider components of the `vertical curvature' $\fnb{\pdeform, \vertproj}$, giving the following result.
\begin{theorem}
\label{Tcurv}
The component of the vertical curvature $\fnb{\pdeform, \vertproj}$ in $V\pi_{1,0}$ satisfies
\begin{equation*}
\vertproj \circ \fnb{\pdeform, \vertproj} + R^\Gamma + \Phi = 0 \, .
\end{equation*}
\end{theorem}
\begin{proof}
This follows readily from the decomposition of the identity as $\vertproj + \pdeform + \horproj$ and the definition of the \FN\ bracket.
\end{proof}
If we use the more refined splitting~(AB), we may find some additional operators. We may decompose $R^\Gamma$ into two components
\begin{equation*}
\Rt^\Gamma = \vertprojt \circ \fnb{\pdeform, \pdeform} \, , \qquad R^\Gamma_+ = \vertproj_+ \circ \fnb{\pdeform, \pdeform}
\end{equation*}
and similarly we may decompose $\Phi$ into two components,
\begin{equation*}
\Phit = \vertprojt \circ \fnb{\pdeform, \horproj} \, , \qquad \Phi_+ = \vertproj_+ \circ \fnb{\pdeform, \horproj}
\end{equation*}
by using in each case the decomposition of $V\pi_{1,0}$ as $(\dist_0 \cap V\pi_{1,0}) \oplus \dist_+$.

We may also write the vertical curvatures $\fnb{\pdeform, \vertproj}$ as a sum $\fnb{\pdeform, \vertprojt} + \fnb{\pdeform, \vertproj_+}$. We shall write $r_+ = \vertproj_+ \circ \fnb{\pdeform, \vertprojt}$, so that
\begin{equation*}
r_+ = \biggl( v^k \biggl(v^p \pd{F^\sigma_{ik}}{y^\nu_1}  - \pd{F^\sigma_{ik}}{y^\nu_p}  
-  (v^p \delta^1_i - \delta^p_i) H_{\nu k}^\sigma \biggr)  - \delta^\sigma_\nu \pd{v^p}{x^i} \biggr) dx^i \wedge \psi^\nu_p \otimes \vf{y^\sigma_1} \, ,
\end{equation*}
and with this new operator we obtain a refinement of Theorem~\ref{Tcurv}, proved in the same way.
\begin{theorem}
The component of the vertical curvature $\fnb{\pdeform, \vertproj}$ in $\dist_+$ satisfies
\begin{equation*}
\vertproj_+ \circ \fnb{\pdeform, \vertproj_+} + R^\Gamma_+ + \Phi_+ + r_+ = 0 \, .
\end{equation*} \qed
\end{theorem}


\section{The second order approach}

In the case of ODEs a second order approach to the tangent bundle decomposition yields the same result as the first order approach. For PDEs, though, the situation is slightly different. We may again start with a nonvanishing closed $1$-form $\varphi$ on $X$ in order to take advantage of the corresponding $\varphi$-vertical endomorphism $S_2^\varphi$, but we shall see that for each system of PDEs and each $1$-form $\varphi$ there is a condition, of a local nature, which must be satisfied in order to construct a unique splitting. If the condition does not hold then we can still find a splitting, but in the absence of further data the construction will not be unique.

We now consider, as before, the projections
\begin{equation*}
J^1\pi \xrightarrow{\pi_{1,0}} Y \xrightarrow{\pi} X
\end{equation*}
and the filtration
\begin{equation*}
V\pi_{1,0} \subset V\pi_1 \subset TJ^1\pi \, ;
\end{equation*}
given a second order connection $\Gamma$ we take, as before, $D_\Gamma$ to be the distribution spanned locally by the vector fields
\begin{equation*}
\Gamma_i = \vf{x^i} + y^\sigma_i \vf{y^\sigma} + F^\sigma_{ij} \vf{y^\sigma_j}
\end{equation*}
with, once again, $F^\sigma_{ij} = y^\sigma_{ij} \circ \Gamma$, so that $TJ^1\pi = D_\Gamma \oplus V\pi_1$. We shall now attempt to define a distribution $D_-$ satisfying $TJ^1\pi = D_- \oplus V\pi_{1,0}$, giving a splitting
\begin{equation*}
TJ^1\pi = (D_- \cap V\pi_1) \oplus D_\Gamma \oplus V\pi_{1,0} \, , \tag{C}
\end{equation*}
in such a way that $D_- \cap V\pi_1 = \dist_-$, so that splittings~(B) and~(C) are identical.

To construct $D_-$ we mimic the approach taken for ODEs. We must, as in the first order approach, assume given a nonvanishing closed $1$-form $\varphi$ on $X$, so that we can use the second order $\varphi$-vertical endomorphism $S_2^\varphi$. The image of the second order connection $\Gamma(J^1\pi)$ may be regarded as a submanifold of $J^2\pi$ defined locally by the equations $y^\sigma_{ij} = F^\sigma_{ij}$,  and so the annihilator $T^\circ (\Gamma(J^1\pi)) \subset T^*_{\Gamma(J^1\pi)} J^2\pi$ of its tangent space is a codistribution along the image, spanned locally by the $1$-forms
\begin{equation*}
dy^\sigma_{ij} - dF^\sigma_{ij} \, .
\end{equation*}
Operating on this codistribution with $S_2^\varphi$ gives a new codistribution along the image, now spanned locally by the $1$-forms
\begin{align*}
\lefteqn{S_2^\varphi(dy^\sigma_{ij} - dF^\sigma_{ij})} \\
& = \tfrac{1}{2} n(ij) \biggl( \pd{\varphi_i}{x^j} +  \pd{\varphi_j}{x^i} \biggr) \omega^\sigma
- \varphi_k  \pd{F^\sigma_{ij}}{y^\nu_k} \omega^\nu +  \varphi_i \omega^\sigma_j  +  \varphi_j \omega^\sigma_i
\end{align*}
where in the calculation we use the relationship
\begin{equation*}
\pd{y^\sigma_{ij}}{y^\nu_{kl}} = \tfrac{1}{2} n(ij) \delta^\sigma_\nu ( \delta^k_i \delta^l_j + \delta^l_i \delta^k_j)
\end{equation*}
to take account of the symmetry $y^\sigma_{ij} = y^\sigma_{ji}$. Continuing as in the ODE case, we take the pullback codistribution $\Gamma^* S_2^\varphi \bigl( T^\circ (\Gamma(J^1\pi)) \bigr) \subset T^* J^1\pi$, spanned locally by the $1$-forms
\begin{align}
\psi^\sigma_{ij} & = \Gamma^* S_2^\varphi(dy^\sigma_{ij} - dF^\sigma_{ij}) \nonumber \\
& = \tfrac{1}{2} n(ij) \biggl( \pd{\varphi_i}{x^j} +  \pd{\varphi_j}{x^i} \biggr) \omega^\sigma
- \varphi_k \pd{F^\sigma_{ij}}{y^\nu_k} \omega^\nu +  \varphi_i \ombar^\sigma_j  +  \varphi_j \ombar^\sigma_i \, , \label{psicoord}
\end{align}
where $\ombar^\sigma_i = dy^\sigma_i - F^\sigma_{ik} dx^k$. Finally we let $D_-$ be the distribution satisfying
\begin{equation*}
D_-^\circ = \Gamma^* S_2^\varphi \bigl( T^\circ \Gamma(J^1\pi) \bigr) \, ;
\end{equation*}
we shall show that $D_\Gamma \subset D_-$, and we shall consider the circumstances where we can write $D_- \oplus V\pi_{1,0} = TJ^1\pi$.

Before doing this, though, we note that the awkwardness of the coordinate expression~(\ref{psicoord}) above can be reduced by using adapted coordinates on $X$ where $\varphi = dx^1$, as in the first order approach. In these adapted coordinates, we will therefore have
\begin{equation*}
\psi^\sigma_{11} = - \pd{F^\sigma_{11}}{y^\nu_1} \omega^\nu + 2 \ombar^\sigma_1 \, , \qquad
\psi^\sigma_{p1} = - \pd{F^\sigma_{p1}}{y^\nu_1} \omega^\nu + \ombar^\sigma_p \, , \qquad
\psi^\sigma_{pq} = - \pd{F^\sigma_{pq}}{y^\nu_1} \omega^\nu \, .
\end{equation*}
To check the properties of $D_-$, suppose that
\begin{equation*}
U = U^i \vf{x^i} + U^\sigma \vf{y^\sigma} + U^\sigma_1 \vf{y^\sigma_1} + U^\sigma_p \vf{y^\sigma_p} \in D_- \, ;
\end{equation*}
then
\begin{subequations}
\label{Brackets}
\begin{align}
- \pd{F^\sigma_{11}}{y^\nu_1} (U^\nu - y^\nu_i U^i) + 2 (U^\sigma_1 - F^\sigma_{1k} U^k) & = 0 \\
- \pd{F^\sigma_{p1}}{y^\nu_1} (U^\nu - y^\nu_i U^i) + (U^\sigma_p - F^\sigma_{pk} U^k) & = 0  \\
- \pd{F^\sigma_{pq}}{y^\nu_1} (U^\nu - y^\nu_i U^i) & = 0 \, . \label{unwanted}
\end{align}
\end{subequations}
We see first that if $U^\sigma = y^\sigma_i U^i$ and $U^\sigma_j = F^\sigma_{ij} U^i$ then all three sets of equations are satisfied, so that $D_\Gamma \subset D_-$; and then we see that if the equations are satisfied and $U^i = U^\sigma = 0$ then $U^\sigma_j = 0$, so that $D_- \cap V\pi_{1,0} = 0$. We must now see whether $\rank D_- = m+n$, so that we can write $TJ^1\pi = D_- \oplus V\pi_{1,0}$.

Suppose that the $1$-form $\varphi$ satisfies the `compatibility condition' that, for any vector field $v$ on $M$ satisfying $i_v\varphi = 0$ and for any vector field $W$ on $J^1\pi$ taking values in $S_1^\varphi(TJ^1\pi)$, the Lie bracket $[W,\Gamma_v]$ must also take its values in $S_1^\varphi(TJ^1\pi)$. If $v = v^i \p / \p x^i$ where $v^i \varphi_i = 0$, and if $W = W^\sigma \varphi_j \p / \p y^\sigma_j$, then
\begin{equation*}
[W,\Gamma_v] = v^p W^\nu \biggl( \pd{F^\sigma_{p1}}{y^\nu_1} \vf{y^\sigma_1} 
+ \pd{F^\sigma_{pq}}{y^\nu_1} \vf{y^\sigma_q} \biggr) - \Gamma_v(W^\nu) \vf{y^\nu_1} \, ,
\end{equation*}
so we see that the compatibility condition is satisfied exactly when $\p F^\sigma_{pq} / \p y^\nu_1 = 0$, in other words when equations~\eqref{unwanted} vanish. We therefore have the following result.
\begin{theorem}
Given a second order connection $\Gamma$ and a nonvanishing closed $1$-form $\varphi$ on $X$ satisfying the compatibility condition, the distribution $D_-$ specified by the annihilator condition $D_-^\circ = \Gamma^* S_2^\varphi \bigl( T^\circ (\Gamma(J^1\pi)) \bigr)$ has rank $m+n$ and gives rise to the tangent bundle decomposition $TJ^1\pi = (D_- \cap V\pi_1) \oplus D_\Gamma \oplus V\pi_{1,0}$. This is the same as splitting~(B) but is independent of any choice of vector field on $X$.
\end{theorem}
\begin{proof}
As the compatibility condition is satisfied, we see that $D_-$ is spanned locally by $D_\Gamma$ together with the $m$ linearly independent vector fields
\begin{equation}
H_\sigma = \vf{y^\sigma} + H^\nu_{\sigma k} \vf{y^\nu_k} \, , \qquad H^\nu_{\sigma k} = \tfrac{1}{2} n(1k) \pd{F^\nu_{1k}}{y^\sigma_1} \, . \label{horvf}
\end{equation}
It is immediate that the distribution spanned by the vector fields $H_\sigma$ is just $\dist_-$ for any compatible choice of vector field $v$.
\end{proof}
There is, however, no guarantee that a given second order connection will have any compatible $1$-forms, and so we must consider how to deal with this possibility. One approach might be to return to the codistribution $D_-^\circ$, spanned locally by the $1$-forms $\psi^\sigma_{ij}$, and replace this by a smaller codistribution spanned locally by some linear combinations of these $1$-forms. For example, if $X$ supports a Riemannian metric (or indeed a conformal class of such metrics) then we may choose a representative metric $g$ normalized with respect to the $1$-form $\varphi$ (so that $\gh(\varphi,\varphi) = 1$ where $\gh$ is the corresponding contravariant metric) and then the $1$-forms
\begin{equation*}
\psi^\sigma_1 = \psi^\sigma_{11} + \tfrac{1}{2} g^{1p} g^{1q} \psi^\sigma_{pq} \, , \qquad
\psi^\sigma_p = \psi^\sigma_{1p} - g^{1q} \psi^\sigma_{pq} 
\end{equation*}
specify a well-defined codistribution $D_g^\circ$ such that $D_g$ has rank $m+n$ and gives rise to a splitting $TJ^1\pi = (D_g \cap V\pi_1) \oplus D_\Gamma \oplus V\pi_{1,0}$. In fact this recovers the splitting~(B) obtained using the first order approach, because given the $1$-form $\varphi$ we are automatically provided with a vector field $v = g^\sharp(\varphi)$, where $g^\sharp$ is the metric operation of `raising the index', so that $v^i = g^{ij} \varphi_j$ in general coordinates, and $v^i = g^{i1}$ in adapted coordinates. If we scale the metric so that $g^{11} = 1$, we see that $i_v\varphi = 1$ and that the distribution $\dist_-$ corresponding to an eigenvalue of $-1$ is now spanned by the vector fields
\begin{gather*}
H_\sigma = \vf{y^\sigma} + H^\nu_{\sigma k} \vf{y^\nu_k} \\
H^\nu_{\sigma 1} = \tfrac{1}{2} \biggl( \pd{F^\nu_{11}}{y^\sigma_1} - g^{1p} g^{1q} \pd{F^\nu_{pq}}{y^\sigma_1} \biggr) \, , \qquad
H^\nu_{\sigma q} = \biggl( \pd{F^\nu_{1q}}{y^\sigma_1} + g^{1p} \pd{F^\nu_{pq}}{y^\sigma_1} \biggr)
\end{gather*}
and is just $D_g \cap V\pi_1$.


\section{Examples}

For our first example we consider a possible generalization of a result about the separability of systems of second order ODEs.

In~\cite{Canetal96} some necessary and sufficient conditions for separability of ODEs were given, two of which were that a curvature $R$ (whose components were essentially commutators of horizontal vector fields) should vanish, and that the Jacobi endomorphism $\Phi$ should be diagonalizable. For the PDE case we shall show that, if a second order connection is separable in the same sense, than similar conditions hold; a topic for future investigation would be to find further conditions which would then be sufficient as well as necessary.

We start with a system of second-order PDEs in the form
\begin{equation*}
\pdb{\phi^\sigma}{x^i}{x^j} = F^\sigma_{ij} \biggl( \phi^\nu, \pd{\phi^\nu}{x^k} \biggr)
\end{equation*}
so that the functions $F^\sigma_{ij}$ do not depend on the variables $x^k$. We suppose that, in these coordinates, the equations are separable in the sense that $F^\sigma_{ij}$ does not depend on $\phi^\nu$ or on $\partial \phi^\nu / \partial x^k$ for $\nu \ne \sigma$.

Let $\varphi = dx^1$ and let $v$ satisfy $v^1 = 1$, with the other components of $v$ as yet undetermined. Considering the vector fields
\begin{equation*}
H_\sigma = \vf{y^\sigma} + H^\nu_{\sigma k} \vf{y^\nu_k}
\end{equation*}
we therefore have $H^\nu_{\sigma k} = 0$ for $\nu \ne \sigma$, and (with implicit sums over repeated indices $p$, $q$ but with no sum over the repeated index $\sigma$)
\begin{align*}
H^\sigma_{\sigma 1} & = \tfrac{1}{2} \biggl( \pd{F^\sigma_{11}}{y^\sigma_1} - v^p v^q \pd{F^\sigma_{pq}}{y^\sigma_1} \biggr) \\
H^\sigma_{\sigma q} & = \pd{F^\sigma_{1q}}{y^\sigma_1} + v^p \pd{F^\sigma_{pq}}{y^\sigma_1} \, .
\end{align*}

We now take `directional slices' in the domain by choosing particular vector fields $v$. Set
\begin{equation*}
v_1 = (1, 0, 0, \ldots, 0) \, , \qquad v_p = (1, 0, \ldots, 0, 1, 0, \ldots, 0)
\end{equation*}
where the nonzero components of $v_p$ are $v_p^1$ and $v_p^p$. Using a superscript in parentheses to indicate the slice, we then have
\begin{equation*}
\slice{1}H^\sigma_{\sigma 1} = \tfrac{1}{2} \pd{F^\sigma_{11}}{y^\sigma_1} \, , \qquad
\slice{p}H^\sigma_{\sigma 1} = \tfrac{1}{2} \biggl( \pd{F^\sigma_{11}}{y^\sigma_1} - \pd{F^\sigma_{pp}}{y^\sigma_1} \biggr)
\end{equation*}
and
\begin{equation*}
\slice{1}H^\sigma_{\sigma q} = \pd{F^\sigma_{1q}}{y^\sigma_1} \, , \qquad
\slice{p}H^\sigma_{\sigma q} = \pd{F^\sigma_{1q}}{y^\sigma_1} + \pd{F^\sigma_{pq}}{y^\sigma_1} \, .
\end{equation*}
Now for any slice $(i)$ we have
\begin{equation*}
[\slice{i}H_\sigma, \slice{i}H_\rho] 
= \bigl( \slice{i}H_\sigma(\slice{i}H^\nu_{\rho k}) - \slice{i}H_\rho(\slice{i}H^\nu_{\sigma k}) \bigr) \vf{y^\nu_k}
\end{equation*}
(sum over $\nu$, $k$ on the right-hand side) so that, as
\begin{equation*}
\slice{i}H_\sigma(\slice{i}H^\nu_{\rho k}) = 0
\end{equation*}
we see that $[\slice{i}H_\sigma, \slice{i}H_\rho] = 0$ and hence that $\slice{i}R^H = 0$. In addition, when considering the Jacobi curvature $\Phi$, we have in general
\begin{equation*}
\Phi^\nu_{i\sigma j} = H^\rho_{\sigma i} H^\nu_{\rho j} + \Gamma_i(H^\nu_\sigma j) - H_\sigma(F^\nu_{ij}) \, ,
\end{equation*}
but for separable systems $\Phi^\nu_{i\sigma j} = 0$ when $\nu \ne \sigma$, so that (with no sum over the repeated index $\sigma$)
\begin{equation*}
\Phi^\sigma_{i\sigma j} = H^\sigma_{\sigma i} H^\sigma_{\sigma j} + \Gamma_i(H^\sigma_\sigma j) - H_\sigma(F^\sigma_{ij}) \, .
\end{equation*}
Thus, for each pair $(i,j)$, $\Phi$ is diagonal with the only nonzero components being $\Phi^\sigma_{i\sigma j}$. In summary, therefore, we have the following result.
\begin{prop}
Let $\Gamma$ be a second order connection and suppose that, in coordinates $(x^k, y^\sigma, y^\sigma_i, y^\sigma_{ij})$, the functions $F^\sigma_{ij} = y^\sigma_{ij} \circ \Gamma$ do not involve the independent variables $x^k$. If the corresponding second order PDEs are separable in these coordinates then $R^H = 0$, and in addition $\Phi$ is diagonal for each coordinate pair $(i,j)$. \qed
\end{prop}

For our second example we consider certain PDEs associated with harmonic maps. If $(M,g)$ and $(N,h)$ are Riemannian manifolds than a map $\phi : M \to N$ is harmonic if its tension field $\tau(\phi)$ vanishes~\cite{HeWo08}. This equation is given in local coordinates $(x^i)$ on $M$ and $(y^\sigma)$ on $N$ as
\begin{equation*}
\tau(\phi)^\rho 
= g^{ij} \biggl( \pdb{\phi^\rho}{x^i}{x^j} - {}^g\Gamma^k_{ij} \pd{\phi^\rho}{x^k} 
+ {}^h\Gamma^\rho_{\sigma\nu}(\phi) \pd{\phi^\sigma}{x^i} \pd{\phi^\nu}{x^j} \biggr) = 0
\end{equation*}
where ${}^g\Gamma^k_{ij}$ and ${}^h\Gamma^\rho_{\sigma\nu}$ denote the Christoffel symbols of $g$ and $h$ respectively. A sufficient condition for $\phi$ to be harmonic is therefore that
\begin{equation}
\label{Harminv}
\pdb{\phi^\rho}{x^i}{x^j} - F^\rho_{ij} = \pdb{\phi^\rho}{x^i}{x^j} - {}^g\Gamma^k_{ij} \pd{\phi^\rho}{x^k} + {}^h\Gamma^\rho_{\sigma\nu}(\phi) \pd{\phi^\sigma}{x^i} \pd{\phi^\nu}{x^j} = 0
\end{equation}
for any index $\rho$ and any symmetric pair of indices $(i,j)$. (In the case of immersions $\phi$ the quantities $\p^2 \phi^\rho / \p x^i \p x^j - F^\rho_{ij}$ are differential invariants~\cite{MuKr03}.)

To analyse equations~\eqref{Harminv} in the present context put $X = M$ and $Y = M \times N$, and let $\pi : Y \to X$ be projection on the first factor. The equations may then be written as
\begin{equation*}
F^\rho_{ij} = {}^g\Gamma^k_{ij}(x) y^\rho_k - {}^h\Gamma^\rho_{\sigma\nu}(y) y^\sigma_i y^\nu_j \, ,
\end{equation*}
in other words as the coordinate equations corresponding to a second order connection $\Upsilon : J^1\pi \to J^2\pi$.

We may apply the first order approach to finding a (local) decomposition of the tangent bundle $TJ^1\pi$ by fixing a point $x \in X$ and taking Riemannian normal coordinates about $x$, letting $x^1$ be the radial coordinate and setting $\varphi = dx^1$. We must also choose a vector field $v$ on $X$, and as mentioned before we choose $v = g^\sharp(\varphi)$, so that $v = \p / \p x^1$. We now have
\begin{align*}
H_\sigma & = \vf{y^\sigma} + \Bigl( \tfrac{1}{2} \delta^\rho_\sigma \; {}^g\Gamma^1_{11}(x) 
- {}^h\Gamma^\rho_{\sigma\nu}(y) y^\nu_1 \Bigr) \vf{y^\rho_1} \\
& \qquad + \Bigl( \delta^\rho_\sigma \; {}^g\Gamma^1_{1q}(x) - {}^h\Gamma^\rho_{\sigma\nu}(y) y^\nu_q \Bigr) \vf{y^\rho_q} \\
& = \vf{y^\sigma} - {}^h\Gamma^\rho_{\sigma\nu}(y)  y^\nu_k \vf{y^\rho_k} 
\end{align*}
where we have used the fact that, in normal coordinates, $g_{11} = 1$ and $g_{1q} = 0$ so that ${}^g\Gamma^1_{1j} = 0$.

We may also compute the curvatures $R^\Gamma$ and $R^H$ (see the Appendix), the Jacobi curvature $\Phi$ and the vertical curvature $r_+$. We find that
\begin{equation*}
R^\Gamma = \curv^\nu_{ijk} dx^i \wedge dx^j \otimes \vf{y^\nu_k} \, , \qquad 
R^H = \curv^\nu_{\sigma\rho k} \omega^\sigma \wedge \omega^\rho \otimes \vf{y^\nu_k}
\end{equation*}
and that
\begin{equation*}
\Phi = \Phi^\sigma_{i\nu k} dx^i \wedge \omega^\nu \otimes \vf{y^\sigma_k} \, , \qquad 
r_+ = r^{p\sigma}_{i\nu} dx^i \wedge \psi^\nu_p \otimes \vf{y^\sigma_1}
\end{equation*}
where
\begin{align*}
\curv^\nu_{ijk} & = {}^g \! R^l_{kij} y^\nu_l + {}^h \! R^\nu_{\rho\mu\sigma} y^\sigma_i y^\mu_j y^\rho_k \\
\curv^\nu_{\sigma\rho k} & = {}^h\! R^\nu_{\mu\rho\sigma} y^\mu_k \\
\Phi^\rho_{i\nu k} & = {}^h \!R^\rho_{\mu\nu\sigma}y^\sigma_i y^\mu_k \\
r^{p\sigma}_{i\nu} & = - \, {}^g\Gamma^p_{i1}(x) \delta^\sigma_\nu
\end{align*}
and where ${}^g \! R^l_{kij}$ and ${}^h \! R^\rho_{\nu\mu\sigma}$ are the Riemannian curvature tensors of the two metrics $g$ and $h$. Note that, for our particular choice of $\varphi$ as the differential of the radial coordinate about some point, both $R^H$ and $\Phi$ are independent of the metric $g$ on $M$, whereas $r_+$ is independent of the metric $h$ on $N$.

For our final example we consider a lemniscate with oscillating amplitude, given by
\begin{equation}
\label{Ex1a}
r^2 = a^2(t) \cos 2\theta \, , \qquad \ddot{a} = -a \, .
\end{equation} 
Here we treat $r$ as the dependent variable and $(t,\theta)$ as independent variables. Differentiating and eliminating $a(t)$ and $\theta$ gives
\begin{equation*}
r_{\theta\theta} = -2r - \frac{r_\theta^2}{r} \, , \qquad r_{t\theta} = \frac{r_t r_\theta}{r} \, , \qquad r_{tt} = -r \, .
\end{equation*}
These three PDEs describe a submanifold of $J^2\pi$ ($\pi : Y \to X$) which properly contains the prolonged solution~(\ref{Ex1a}). In terms of the second order connection $\Gamma$ we therefore have
\begin{equation*}
\Gamma_{\theta\theta} = r_{\theta\theta} \circ \Gamma = - 2r - \frac{r_\theta^2}{r} \, , \qquad
\Gamma_{t\theta} = r_{t\theta} \circ \Gamma = \frac{r_t r_\theta}{r} \, , \qquad 
\Gamma_{tt} = r_{tt} \circ \Gamma = -r
\end{equation*}
so that the vector fields $\Gamma_1$, $\Gamma_2$ are given by
\begin{align*}
\Gamma_1 & = \vf{t} + r_t \vf{r} - r \vf{r_t} + \frac{r_t r_\theta}{r} \vf{r_\theta} \\
\Gamma_2 & = \vf{\theta} + r_\theta \vf{r} + \frac{r_t r_\theta}{r} \vf{r_t} - \biggl( 2r + \frac{r_\theta^2}{r} \biggr) \vf{r_\theta} \, . 
\end{align*}
The curvature $R^\Gamma$ is given by
\begin{equation*}
R^\Gamma = dt \wedge d\theta \otimes [\Gamma_1, \Gamma_2] \, ;
\end{equation*}
as
\begin{equation*}
[\Gamma_1, \Gamma_2] = \bigl( \Gamma_1(\Gamma_{\theta t}) - \Gamma_2(\Gamma_{tt}) \bigr) \vf{r_t}
+ \bigl( \Gamma_1(\Gamma_{\theta\theta}) - \Gamma_2(\Gamma_{t\theta}) \bigr) \vf{r_\theta}
= 0 \, ,
\end{equation*}
we see that $R^\Gamma$ vanishes.

The other curvatures involve specific choices of the $1$-form $\varphi$ and vector field $v$. We shall compute the curvature components $\Phi^\rho_{i\sigma j}$ where necessarily $\rho = \sigma = r$, so that we may write these components as $\Phi_{ij}$, and they are given by
\[
\Phi_{ij} = H^r_{ri} H^r_{rj} + \Gamma_i(H^r_{rj}) - H_r(\Gamma_{ij}) \, .
\]
If we take $\varphi = dt$ and $v = \partial / \partial t$ then
\begin{equation*}
H^t_r = \vf{r} + \tfrac{1}{2} \pd{\Gamma_{tt}}{r_t} \vf{r_t} + \pd{\Gamma_{t\theta}}{r_t} \vf{r_\theta}
= \vf{r} + \frac{r_\theta}{r} \vf{r_\theta} \, ,
\end{equation*}
so that in this case $H^r_{rt} = 0$ and $H^r_{r\theta} = r_\theta / r$; if instead we take $\varphi = d\theta$ and $v = \partial / \partial \theta$ then
\begin{equation*}
H^\theta_r = \vf{r} + \pd{\Gamma_{\theta t}}{r_\theta} \vf{r_t} + \tfrac{1}{2} \pd{\Gamma_{\theta\theta}}{r_\theta} \vf{r_\theta}
= \vf{r} + \frac{r_t}{r} \vf{r_t} - \frac{r_\theta}{r} \vf{r_\theta} \, ,
\end{equation*}
so now $H^r_{rt} = r_t / r$ and $H^r_{r\theta} = - r_\theta / r$. Writing $\Phi^{(t)}_{ij}$ and $\Phi^{(\theta)}_{ij}$ to indicate the choice of $\varphi$ and $v$ we find after some short calculations that
\begin{equation*}
\Phi^{(t)}_{tt} = 1 \, , \qquad \Phi^{(t)}_{t\theta} = \Phi^{(t)}_{\theta t} = \Phi^{(t)}_{\theta\theta} = 0 \, ,
\end{equation*}
and
\begin{equation*}
\Phi^{(\theta)}_{tt} = \Phi^{(\theta)}_{t\theta} = \Phi^{(\theta)}_{\theta t} = 0 \, , \qquad \Phi^{(\theta)}_{\theta\theta} = 4 \, .
\end{equation*}
These curvature components will be of use in the investigation of singularity formation in future work.

\section*{Appendix: curvature calculations}

We consider first the three-fold splitting~(B) with $TJ^1\pi = \dist_- \oplus D_\Gamma \oplus V\pi_{1,0}$, with projectors $\horproj$, $\pdeform$ and $\vertproj$. The various \FN\ brackets are given in general coordinates by
\begin{align*}
\fnb{\pdeform, \pdeform} & = R^\Gamma = dx^i \wedge dx^j \otimes [\Gamma_i, \Gamma_j] \\
\fnb{\pdeform, \horproj} & = \Phi + dx^i \wedge \psi^\sigma_i  \otimes H_\sigma \\
\fnb{\pdeform, \vertproj} & = - \Phi - dx^i \wedge \psi^\sigma_i \otimes H_\sigma  -  dx^i \wedge dx^j \otimes [\Gamma_i, \Gamma_j] \\
\fnb{\horproj, \horproj} & = \omega^\sigma \wedge \omega^\nu \otimes [H_\sigma, H_\nu]  - 2 dx^i \wedge \psi^\sigma_i \otimes H_\sigma \\
\fnb{\horproj, \vertproj} & =  - \Phi  -   \omega^\sigma \land \omega^\nu \otimes  [H_\sigma, H_\nu]  + dx^i \wedge \psi^\sigma_i  \otimes H_\sigma \\
\fnb{\vertproj, \vertproj} & = 2 \Phi + dx^i \wedge dx^j \otimes [\Gamma_i, \Gamma_j]   + \omega^\nu \wedge \omega^\rho \otimes  [H_\nu, H_\rho] \, ,
\end{align*}
where the Lie brackets of the vector fields in the expressions above are given by
\begin{align*}
[\Gamma_i, \Gamma_j] & = \bigl( \Gamma_i (F^\nu_{jk}) - \Gamma_j(F^\nu_{ik}) \bigr) \vf{y^\nu_k} \\
[H_\sigma, H_\nu] & = \bigl( H_\sigma (H^\rho_{\nu k}) - H_\nu(H^\rho_{\sigma k} ) \bigr) \vf{y^\rho_k} \\
[\Gamma_i, H_\sigma] & = -  H^\nu_{\sigma i}H_\nu + \Phi _{i \sigma} \\
\biggl[ \Gamma_i, \vf{y^\sigma_j} \biggr] 
& = - \delta_i^j H_\sigma + \biggl( \delta_i^j H_{\sigma k}^{\nu}  - \pd{F^\nu_{ik}}{y^\sigma_j} \biggr) \vf{y^\nu_k} \\
\biggl[ H_\sigma, \vf{y^\nu_j} \biggr] & = - \pd{H^\rho_{\sigma k} }{y^\nu_j}  \vf{y^\rho_k}
\end{align*}
and where
\begin{equation*}
\Phi = dx^i \wedge \omega^\sigma \otimes \Phi_{i\sigma} = \Phi^\nu_{i \sigma j} dx^i \wedge \omega^\sigma \otimes \vf{y^\nu_j} \, , \qquad
\Phi^\nu_{i \sigma j} = H^\rho_{\sigma i} H^\nu_{\rho j} + \Gamma_i(H^\nu_{\sigma j}) - H_\sigma(F^\nu_{ij})
\end{equation*}
is the Jacobi curvature.

For the four-fold splitting~(AB) we use adapted coordinates with $\varphi = dx^1$ and projectors $\horproj$, $\pdeform$, $\vertprojt$, $\vertproj_+$. The Jacobi curvature may be decomposed into two partial Jacobi curvatures
\begin{equation*}
\Phi = \Phit + \Phi_+ = (\vertprojt \circ \Phi) + (\vertproj_+ \circ \Phi) 
\end{equation*}
where
\begin{align*}
\Phit & = \Phi^\nu_{i \sigma p} dx^i \wedge \omega^\sigma \otimes W^p_\nu \, , \qquad W^p_\nu = \vf{y^\nu_p} - v^p \vf{y^\nu_1}\\
\Phi_+ & = v^j \Phi^\nu_{i \sigma j}  dx^i \wedge \omega^\sigma \otimes \vf{y^\nu_1} \, .
\end{align*}
In a similar way we may decompose the curvature of the second order connection as
\begin{equation*}
R^\Gamma = \Rt^\Gamma + R_+^\Gamma = (\vertprojt \circ R^\Gamma) + (\vertproj_+ \circ R^\Gamma) 
\end{equation*}
where
\begin{align*}
\Rt^\Gamma & = [\Gamma_i, \Gamma_j]^\nu_p \, dx^i \wedge dx^j  \otimes W_\nu^p \\
R^\Gamma_+ & = v^k [\Gamma_i, \Gamma_j]^\nu_k \, dx^i \wedge dx^j  \otimes  \vf{y^\nu_1} \, .
\end{align*}
Also, writing
\begin{equation*}
R^H = \vertproj \circ  [\horproj, \horproj] = \omega^\sigma \wedge \omega^\rho \otimes [H_\sigma, H_\rho] \, , 
\end{equation*}
we may put
\begin{equation*}
R^H = \Rt^H + R_+^H = (\vertprojt \circ R^H) + (\vertproj_+ \circ R^H) 
\end{equation*}
giving
\begin{align*}
\Rt^H & = [H_\sigma, H_\rho]^\nu_p \, \omega^\sigma \wedge \omega^\rho \otimes W_\nu^p \\
R^H_+ & = v^k [H_\sigma, H_\rho]^\nu_k \, \omega^\sigma \wedge \omega^\rho \otimes  \vf{y^\nu_1} \, .
\end{align*}
We also have
\begin{equation*}
\vertproj_+ \circ \fnb{\pdeform, \vertproj_+} = -R^\Gamma_+ - \Phi_+ - r_+
\end{equation*}
where, as before, $r_+ = \vertproj_+ \circ \fnb{\pdeform, \vertprojt}$, so that
\begin{equation*}
r_+ = \biggl( v^k \biggl(v^p \pd{F^\sigma_{ik}}{y^\nu_1}  - \pd{F^\sigma_{ik}}{y^\nu_p}  
-  (v^p \delta^1_i - \delta^p_i) H_{\nu k}^\sigma \biggr)  - \delta^\sigma_\nu \pd{v^p}{x^i} \biggr) dx^i \wedge \psi^\nu_p \otimes \vf{y^\sigma_1} \, .
\end{equation*}

\section*{Acknowledgements}

Research supported by grant no 14-02476S `Variations, Geometry and Physics' of the Czech Science Foundation and IRSES project GEOMECH (EU FP7, nr 246981). Two of us (DS, OR) wish to acknowledge the hospitality of the Australian Mathematical Sciences Institute.

\end{document}